\documentclass{amsart}
\usepackage{xcolor}

\newtheorem{thm}{Theorem}
\newtheorem{theorem}[thm]{Theorem}

\newtheorem{corollary}[thm]{Corollary}

\newtheorem{lemma}[thm]{Lemma}

\newtheorem{proposition}[thm]{Proposition}

\theoremstyle{definition}

\newtheorem{definition}[thm]{Definition}

\theoremstyle{remark}

\newtheorem{claim}{Claim}

\usepackage{amssymb}
\usepackage{amsrefs}

\usepackage{hyperref}

\DefineSimpleKey{bib}{primaryclass}{}
\DefineSimpleKey{bib}{archiveprefix}{}

\BibSpec{arXiv}{%
  +{}{\PrintAuthors}{author}
  +{,}{ \textit}{title}
  +{}{ \parenthesize}{year}
  +{,}{ arXiv }{eprint}
  +{,}{ primary class }{primaryclass}
}

\begin{document}

\title{New examples in the study of selectively separable spaces}
\author{Alan Dow and Hayden Pecoraro}

\address{Department of Mathematics and Statistics,
University of North Carolina at Charlotte, 
Charlotte, NC 28223}
\email{adow@charlotte.edu}
\email{hpecoraronsw@gmail.com}



\thanks{The authors thank L. Zdomskyy for sharing corrections to the first version of the paper}

\date{\today}

 \begin{abstract} The property of being selectively separable is
 well-studied and generalizations such as H-separable and wH-separable
 have also generated much interest. Bardyla, Maesano, and Zdomskyy proved
 from Martin's Axiom that there are countable regular wH-separable spaces that
 are not H-separable. We prove there is a ZFC example. Their example was
 also Fr\'echet-Urysohn, and we produce two additional examples from 
 weaker assumptions.   
 \end{abstract}

\maketitle

\section{Introduction}

In this paper we prove the following results.

\begin{theorem}  There is a countable Hausdorff 0-dimensional wH-separable space
that is not H-separable.
\end{theorem}

\begin{theorem}[$\mathfrak b = \mathfrak c$] There is a countable Hausdorff 0-dimensional
 Fr\'echet-Urysohn space that is not H-separable and has $\pi$-weight $\mathfrak c$.
\end{theorem}

\begin{theorem}[$\mathfrak p = \mathfrak b$] There is a countable Hausdorff 0-dimensional 
 Fr\'echet-Urysohn space that is not H-separable and has $\pi$-weight $\mathfrak b$.
\end{theorem}

\begin{corollary}[$\mathfrak c \leq \omega_2$] There\label{cor} is a countable Hausdorff 0-dimensional
Fr\'echet-Urysohn space that is not H-separable and has $\pi$-weight $\mathfrak b$.
\end{corollary}

A space $X$ is Fr\'echet-Urysohn if for every $A\subset  X$ and $x\in \overline{A}\setminus A$,
there is a sequence $S\subset A$ converging to $x$. Every separable Fr\'echet-Urysohn space
is selectively separable \cite{BaDo}. It was observed in \cite{BaMaZd}
that it followed from the results in \cite{GarySakai} that  every 
separable Fr\'echet-Urysohn space is wH-separable.  Corollary \ref{cor}
is included because it  is not known if, in ZFC,
 there is a countable Fr\'echet-Urysohn space with $\pi$-weight exactly 
  $\mathfrak b$. 

In \cite{GarySakai},  
the modifier wH-separable was used to indicate a
\textit{weakening} of H-separable (also introduced in \cite{BBM}
and named for Hurewicz), but we will follow \cite{BaMaZd} and use mH-separable (for \textit{monotone})
as the term for wH-separable.

 \begin{definition} A space $X$ is M-separable, respectively H-separable,
 if for every sequence $\{ D_n : n\in\omega\}$
  of dense subsets, there is a selection $F_n\subset D_n$ ($n\in\omega$) of finite sets
  satisfying that $\bigcup\{ F_n :n\in\omega\}$ is dense in $X$, respectively, for every
   infinite $A\subset \omega$, $\bigcup\{ F_n : n\in A\}$ is dense in $X$.
 \end{definition}

 A space $X$ is mH-separable if it satisfies the conclusion for H-separable for 
 every monotone descending sequence of dense subsets of $X$.  If $\{ D_n :
 n\in\omega\}$ is a sequence of dense sets in a  countable Fr\'echet-Urysohn
 space $X$  then for every $x\in  X$, there is a  
 selection $d^x_n\in D_n$ (for $n\in\omega$) such that $\{ d^x_n :n\in\omega\}$
 has a subsequence   that 
 converges to $x$.  
 If the sequence $\{ D_n : n\in\omega\}$
 is descending then the selector sequence $\{ d^x_n : n\in\omega\}$  
 can be chosen so that the full sequence converges to $x$. 
 If the full selector sequence can be chosen to  converge to $x$,
  for every $x\in X$, then we obtain the H-separable conclusion
  for the sequence $\{ D_n : n\in\omega\}$. 
  This is the essence of why countable Fr\'echet-Urysohn spaces 
  are mH-separable, but can fail to be H-separable.

In this paper we partially answer Question 1.3 of \cite{BaMaZd} by constructing, in ZFC,
an
example of regular countable  mH-separable space that is not H-separable. However
the question remains if there is a
Fr\'echet-Urysohn example.  We do however obtain Fr\'echet-Urysohn examples from
 weaker hypotheses than those in \cite{BaMaZd}.

\section{The main construction}

    In this section we introduce
 some basic ideas that are used.  This introduction should help motivate
 and make the final construction more natural. 
 In each of the main results we construct, by recursion, a crowded 
 clopen  basis    $\tau_{\mathfrak b} = \bigcup_{\alpha<
 \mathfrak b}\tau_\alpha
 $ on $\mathbb Q$ containing
 some  standard countable clopen base  $\tau_0$ for the rational
 topology on $\mathbb Q$.  By the recursion the sequence
  $\{\tau_\alpha : \alpha < \mathfrak b\}$ will be an increasing
  chain of clopen bases such that $\tau_\alpha$ has cardinality
  at most $|\alpha+\omega|$. 
 Let us also fix an enumeration
  $\{ q_n : n\in\omega\}$ of $\mathbb Q$.
  
 One such basic idea is that some  of these properties can be localized.
  Clearly a space $X$ is mH-separable if for every point $x\in X$
  and descending sequence $\{ D_n : n\in\omega\}$ of dense subsets of 
   $X$, there is a sequence of finite sets $H^x_n\subset D_n$ satisfying
   that for every  open $x\in W$,
$W\cap H^x_n$ is not empty for all but finitely many $n\in\omega$. 

It is well-known that a  separable space with character (in fact $\pi$-character) 
less than $\mathfrak b$
 is mH-separable. Let us combine this idea with a new idea for our results.
  
 For each $q\in \mathbb Q$, fix a strictly descending
  local base $\{ U(q,n) :n\in \omega\}\subset \tau_0$ for $q$ with respect
  to $\tau_0$. By a slight abuse of notation 
    let $\tau_0$ be the actual set $\{ \{U(q,n):n\in\omega\} : q\in\mathbb Q\}$
    so indexed.
  
  \begin{definition}
  For any $f\in\omega^\omega$
  and $q\in \mathbb Q$, define
  $$I(q,f) = \bigcup 
   \{ \{q_\ell : \ell<f(n)\}\cap \left( U(q,n)\setminus U(q,n{+}1)\right) :
   n\in\omega\}~.$$
   For convenience, also let  $I(q,f,n)$ be 
   the finite set $I(q,f)\cap \left(U(q,n)\setminus U(q,n{+}1)\right)$.
   \end{definition}

Recall that $H(\mathfrak c^+)$ is the set of sets whose transitive
closure have cardinality at most $\mathfrak c$.  See  \cite{Kunen}
for the definition and for the verification of the following
minor strengthening   of the Lowenheim-Skolem theorem.

\begin{proposition} For any regular cardinal $\kappa\leq \mathfrak c$
and\label{getModels}
any $S\subset H(\mathfrak c^+)$ of cardinality less than $\kappa$,
there is an elementary submodel $M$ of $H(\mathfrak c^+)$ satisfying
that $S\subset M$, $|M|<\kappa$, and $M\cap \kappa$ is an element of $\kappa$.
\end{proposition}
   
   For any sufficiently large (mod finite) $f\in\omega^\omega$,
   $I(q,f)$ is  a sequence that converges to $q$ with respect
   to $\tau_0$. For any  $q\in \mathbb Q$, $f\in\omega^\omega$, and $\alpha<\mathfrak b$,
    we can consider $\{q\}\cup I(q,f)$ as a subspace of $
    (\mathbb Q, \tau_\alpha)$ which will have character less than $\mathfrak b$. 
 One approximation to one of our main innovations is that for selected
 $f_\alpha\in\omega^\omega$, the subspace topology on $\{q\}\cup
  I(q,f_\alpha)$ with respect to $\tau_\alpha$ will (almost) agree
  with the subspace 
   topology on $\{q\}\cup
  I(q,f_\alpha)$ with respect to $\tau_\beta$ for all $\alpha
  \leq\beta< \mathfrak b$. 

  \begin{proposition} If $\{ D_n : n\in\omega\}$ is a  descending
  sequence of $\tau $-dense subsets of $\mathbb Q$
  for some clopen basis $\tau\supset\tau_0$ with $|\tau|<\mathfrak b$,
  there is a sufficiently
  large $f\in\omega^\omega$, such that, for each $q\in \mathbb Q$ and
  $q\in W\in \tau$,  $W\cap H^q_n $ is not empty
  for all but finitely many $n\in\omega$,
   where $H^q_n$ is the finite set $  I(q,f)\cap D_n \cap 
   U(q,n)\setminus U(q, {f(n)})$
   for all $n\in\omega$.    
  \end{proposition}

 \begin{proof}
We intend to make heavy use of elementary submodels so we start here.
Let  $\tau_0,
\tau $ and $\{ D_n : n\in\omega\}$ all be elements of an
elementary submodel $M$ of $H(\mathfrak c^+)$ and assume that
$\tau \subset M$ and 
$|M|<\mathfrak b$. 
Choose any $f\in\omega^\omega$ satisfying that $g<^* f$ for all 
 $g\in M\cap \omega^\omega$.  Consider any  $q\in \mathbb Q$
 and $q\in
 W\in \tau $. By assumption $q$ is not $\tau $-isolated
 and $\tau_0\subset\tau $, hence, for infinitely
 many $n\in\omega$,  
   $W\cap U(q,n)\setminus U(q,n{+}1)$
 is non-empty and clopen. 
 Therefore,
  there is a function $g_W\in M\cap \omega^\omega$  
  satisfying that, for all $n\in\omega$,
   there are $n\leq m,\ell<g_W(n)$ such 
   that $ q_\ell\in D_n\in U(q,n)\setminus U(q,m)$. 
Now set, for $n\in\omega$, $H^q_n = I(q,f)\cap D_n\cap
  U(q,n)\setminus U(q, {f(n)})$. It follows that, for all 
   $q\in W\in \tau$ and $n\in\omega$ such
   that $g_W(n)<f(n)$, 
   $W\cap H^q_n\neq\emptyset$ is non-empty. 
 \end{proof}

Requiring that, for each $\alpha<\mathfrak b$, there is
an $f_\alpha\in\omega^\omega$ so that, for each $q\in \mathbb Q$,
 the $\tau_\beta$-subspace topology on $I(q,f_\alpha)$ equals
 the $\tau_\alpha$-subspace topology for all $\alpha \leq \beta <\mathfrak b$,
  is indeed too strong. For example  it would follow that the $\tau_{\mathfrak b}$
  topology
   would be an $\alpha_1$-space as in \cite{Arh1}. On the other
   hand it is consistent that Fr\'echet-Urysohn $\alpha_1$-spaces
   are first countable (see \cite{DowSteprans}). 

   However the next basic result helps  capture
    the   sense in which the  $\tau_{\mathfrak b}$-subspace topology
   on $\{ q\}\cup I(q,f)$ should   almost agree with   
   the $\tau_\alpha$-subspace topology.  It is influenced
   by the classical fact that a $\psi$-like space built from
   an almost disjoint family on $\omega$ is Fr\'echet-Urysohn
   so long as the family is nowhere maximal.

   \begin{proposition} 
   Assume    that $\varphi$ is a\label{ringRule}  
   function from $I(q,f)$ to $2^{<\omega}$
and let $ \tau\supset\tau_0$ be a clopen base 
for a topology on $\mathbb Q$. Assume that $q$ is in 
the closure of some $A\subset I(q,f)$ with respect to
$\tau$. Let $\sigma\supset  \tau$ be another
clopen base for a topology on $\mathbb Q$ that satisfies
    for each $q\in W\in \sigma$, there is a finite
   subset $\{ \rho_i : i\in\ell\}\subset 2^\omega$ 
   and a $U\in  \tau $,
   $$W\cap I(q,f) =^* U\cap I(q,f) \setminus
     \{ \varphi^{-1}(s) : s\in \{ \rho_i \restriction n: i\in\ell, \ n\in\omega\}~\}~.$$
Then  $q$ is in the $\sigma$-closure of $A$ unless there is some $U\in\tau$
such that 
$T_{A\cap U} = \{  s\in 2^{<\omega} :
 (\exists a \in A\cap U) s \leq \varphi(a)  \}$ has only finitely many
infinite branches.
   \end{proposition}

\begin{proof}   The proof is straightforward. Suppose that $q\in W\in\sigma$
and $W\cap A$ is empty. Choose $U\in \tau$
and finite $\{ \rho_i : i<\ell\}\subset 2^\omega$,
such   that
 $W\cap I(q,f) =^* U\cap I(q,f)\setminus\{ \varphi^{-1}(s): 
  s\in \{ \rho_i \restriction n: i\in\ell, \ n\in\omega\}~\}$. 
  It follows that  $T_{A\cap U} $ is a subset of
  $ \{ \rho_i \restriction n: i\in\ell, \ n\in\omega\}$. 
 \end{proof}

For a set $A\subset\mathbb Q$ and clopen basis $ \tau\supset
\tau_0$, 
let
  $A^{(1)}$ be the set $A$ together with all points
   $q$ that are limits of $  \tau$-converging
   sequences that are contained in $A$. We recall from
   \cite{DowFrechet} that if
   $ \tau$ has cardinality less than
    $\mathfrak b$,  then $\left( A^{(1)}\right)^{(1)}
     = A^{(1)}$.

This next result introduces our main recursive step for constructing
topologies.

\begin{lemma} Suppose that $\tilde \tau \supset \tau_0$ is a
clopen base\label{mainLemma} for a topology on $\mathbb Q$ and let $\mathcal I$
be a family of $\tilde \tau $-converging sequences with $
|\tilde \tau | + |\mathcal I|<\mathfrak b$. Let $\tilde f$ be any
member of $\omega^\omega$.

Let $B\subset \mathbb Q$ be any set that is almost disjoint from each $I\in \mathcal I$.
Also let $A$ be any subset of $\mathbb Q$ and suppose that $r\in \mathbb Q\setminus A$
is not the limit of any  $\tilde \tau $-converging sequence contained in $A$.
 \medskip

 Then there is a function $\tilde f\leq f\in \omega^\omega$ and a
 clopen base $\sigma\supset \tilde \tau$ such that
 \begin{enumerate}
 \item   $I\subset^* I(q,f)$ for each $q\in \mathbb Q$ and $I\in\mathcal I$ that converges
 to $q$,
 \item $\{U_n: n\in\omega\} = \sigma\setminus \tilde \tau$ is a countable partition of $\mathbb Q$,
 \item every element of $\mathcal I$ is a $\sigma$-converging sequence,
 \item $B$ is closed and discrete with respect to $\sigma$,
 \item $r\in U_0$  and $U_0$ is disjoint from  $A$,
 \item for each $0<n\in\omega$ and $ q\in   U_n$,  
    $U_n\cap I(q,f) =^* I(q,f)\setminus B$,
  \item for $q\in U_0 $, 
   $U_0\cap I(q,f) =^* I(q,f) \setminus (B\cup A^{(1)})$.
 \end{enumerate}
 \end{lemma}

   \begin{proof}
   Choose, as in Proposition \ref{getModels},
   an elementary submodel $M$ of $H(\mathfrak c^+)$ of cardinality
    less than $\mathfrak b$ such that $\{\tau_0,\tilde \tau \}\in M$
    and $\tilde \tau$ and $\mathcal I$ are subsets of $M$. 
    Let $f$ be any element of $\omega^\omega$ that mod finite  dominates
     $M\cap \omega^\omega$. 

     \begin{claim} 
     $I\subset^* I(q,f)$ 
     for each $q\in \mathbb Q$ and $I\in\mathcal I$ that
     converges to $q$. 
     \end{claim}

\bgroup
\def\proofname{Proof of Claim.}

\begin{proof}  Let $I\subset \mathbb Q$ be a sequence that converges
to $q$. For each $n\in\omega$, $I\cap U(q,n)\setminus U(q,n+1)$ is finite. 
Therefore, for each $I\in \mathcal I$ that converges to $q$,
 there is a $g_I\in \omega^\omega\cap M$ satisfying that
  $I\subset I(q,g_I)$. Note that it is immediate
  that $I(q,g_I)\subset^* I(q,f)$. 
\end{proof}

For each $q\in A^{(1)}$, let $I_q = I(q,f)\setminus B$,
 and for each $q\notin A^{(1)}$, let $I_q = I(q,f)\setminus
  (B\cup A^{(1)})$.  Notice that $I \subset^* I_q$
  for each $q\in \mathbb Q$ and $I\in \mathcal I$ that
   converges to $q$. 

The family $\{ I_q : q\in \mathbb Q\}$ is an almost disjoint
family, and so, by a finite modification of each,  we
 may assume that the 
 elements of   $\{ I_q : q\in \mathbb Q\}$ are
pairwise disjoint.

     Now we are ready to construct our 
     partition $\{ U_n : n\in\omega\} = 
     \sigma\setminus \tilde \tau$. To do so
     we will recursively define a function $h: \mathbb Q
     \rightarrow \omega$, and will set $U_n = h^{-1}(n)$
     for each $n\in\omega$. The recursion will ensure
     that for each $q\in \mathbb Q$, $I_q$ is almost
     contained in $h^{-1}(n)$ if $h(q)=n$. 
Let us choose any $h_0 : \{r\}\cup B \rightarrow \omega$
that is 1-to-1 with $h_0(r) = 0$. We   
define $h_k$ ($k\in\omega$)
by recursion on $k$ with these three inductive hypotheses
for all $\ell<k$:
\begin{enumerate}
\item $h_\ell\subset h_k$ for all $\ell<k$ and
  the domain of $h_k$ is equal to 
  $\{r\}\cup B \cup \{ q_\ell : \ell<k\}\cup 
    \bigcup\{  I_{q_\ell} : \ell < k\}
    $,
    \item    
     $I_{q_\ell} $ is mod finite contained
     in $h_k^{-1}(h_k(q_\ell))$, and
      \item   if $q_\ell\in A^{(1)}$, $h_k(q_\ell)>0$.
      \end{enumerate}

     Suppose that, by induction, we  have defined
      $h_k$.  We first consider $q_k$. 
    Observe that $I_q\cap \operatorname{dom}(h_k)$ is finite
    for all $q\notin \{r\}\cup \{ q_\ell : \ell < k\}$. 
      If $q_k$ is
      not in the domain of $h_k$, we extend $h_k$ by
      defining $h_{k+1}(q_k) = 1$
      and define $h_{k+1}(q) = 1$ for all $q\in I_{q_k}\setminus
       \operatorname{dom}(h_k)$.
       Otherwise let $n=h_k(q_k)$. If $n>0$,
        then define the extension
        $h_{k+1}(q) = n$ for all $q\in I_{q_k}\setminus
        \operatorname{dom}(h_k)$. It is immediate
        that in these first two cases each of the three inductive
        hypotheses hold for the definition of $h_{k+1}$.
The final case is that $h_k(q_k) = 0$ and we note that
this implies that $q_k\notin A^{(1)}$ and that 
 $I_{q_k}$ is also disjoint from $A^{(1)}$. Therefore the induction
 hypotheses are again preserved if we extend $h_k$ 
 by setting $h_{k+1}(q) = 0$
 for all $q\in I_{q_k}\setminus \operatorname{dom}(h_k)$.

\egroup
     
    Naturally we set $h=\bigcup_k h_k$ and consider
     $U_n = h^{-1}(n)$ for any $0<n\in\omega$. Choose
     any $q\in U_n$ and minimal $k\in\omega$ so that
      $q\in \operatorname{dom}(h_{k+1})$.    
Since $I_q\setminus h_{k+1}^{-1}(n)$ is finite,
 it follows that $I\subset^* I_q \subset^* U_n$ for
 all $I\in\mathcal I$ that converges to $q$. For
 any integer $n>0$, $U_n \cap I(q,f) $ mod finite
 contains $I_q$ and, since $I_q$ mod finite equals
  $I(q,f)\setminus B$, $U_n\cap I(q,f)=^* I(q,f)\setminus B$.
  Similarly, it follows that   $U_0\cap I(q,f) =^* I(q,f)
  \setminus (B\cup A^{(1)}$ for all $q\in U_0$. 

For each $n\in\omega$, $U_n\cap B$ has at most one element,
 hence $B$ is closed and discrete with respect to $\sigma$.
 Finally, $r\in U_0$ and $U_0$ is disjoint from $A^{(1)}\supset A$.
   \end{proof}

\section{The examples}

For the remainder of the paper fix a partition,
 $\{ Q_t : t\in 2^{<\omega}\}$, of $\mathbb Q$ into dense
 sets. This provides us with our function $\varphi$ from
  $\mathbb Q$ to $2^{<\omega}$ where $\varphi(q) = t$
  if and only if $q\in Q_t$.  Fix also an enumeration
   $\{ f_\alpha : \alpha < \mathfrak b\}$
    of a  mod finite unbounded family of strictly
    increasing functions from $\omega^\omega$. 
We begin our discussion with our simple plan for
ensuring that $\tau_{\mathfrak b}$ is not H-separable.

\begin{lemma}  Assume that $\tau_{\mathfrak b}\supset \tau_0$
is a clopen basis\label{notH} for a topology on $\mathbb Q$ that satisfies
\begin{enumerate}
\item  for every $t\in 2^{<\omega}$, $Q_t$ is a $\tau_{\mathfrak b}$-dense
 set,
 \item for every $\alpha\in \mathfrak b$, there is a $\rho_\alpha\in 2^\omega$
 such that $B_\alpha = \bigcup\{
  Q_{\rho\restriction n}\cap \{q_\ell : \ell < f_\alpha(n) \}\}$ 
  is closed and discrete,
\end{enumerate}
then $(\mathbb Q, \tau_{\mathfrak b})$ is not H-separable.
\end{lemma}

\begin{proof}  By the assumption on $\tau_{\mathfrak b}$,
 the family $\{ Q_t : t\in 2^{<\omega}\}$ is a sequence 
 of dense subsets of $(\mathbb Q, \tau_{\mathfrak b})$. 
 We show that $(\mathbb Q, \tau_{\mathfrak b})$
 is not H-separable by considering an arbitrary
selection
 $\{ H_t : t\in 2^{<\omega}\}$  of finite
 sets with $H_t\subset Q_t$ for all $t\in 2^{<\omega}$,
 and proving there is an infinite set $A\subset 2^{<\omega}$
 satisfying that $\bigcup\{ H_t : t\in A\}$ is not dense. 
  
  Define a function $h\in \omega^\omega$ by the formula
 that, for each $n\in\omega$, 
   $\bigcup\{ H_t : t\in 2^n\}$ is a subset of
    $\{ q_\ell : \ell < h(n)\}$. 
    Choose any $\alpha$ so that 
      $L = \{ n : h(n) < f_\alpha(n)\}$ is infinite,
      and let $A  = \{ \rho_\alpha\restriction n : n\in L\}$.
   It follows that the set $  
    \bigcup \{ H_t : t\in A \}$
    is a subset of $B_\alpha$ and is therefore
    is a proper $\tau_{\mathfrak b}$-closed subset of $\mathbb Q$.
    \end{proof}

 Now, for this first construction,
  fix an enumeration $\{ (r_\alpha, A_\alpha) : \alpha < \mathfrak c\}$
 of $\mathbb Q\times [\mathbb Q]^{\aleph_0}$.   
    This   construction
  is the easiest because we can take $\mathfrak c$ many steps. 

\begin{theorem}
If   $\mathfrak b =  \mathfrak c$, 
 there is a countable Hausdorff 0-dimensional  Fr\'echet-Urysohn space that is
 not H-separable.
\end{theorem}

\begin{proof}
We recursively construct an increasing family $\{ \tau_\alpha : \alpha < \mathfrak c\}$
of clopen bases for topologies on $\mathbb Q$. We additionally construct
an increasing family $\{ \mathcal I_\alpha : \alpha < \mathfrak c\}$ 
of families of $\tau_0$-converging infinite
sequences from $\mathbb Q$ and a family
 $\{ \rho_\alpha : \alpha < \mathfrak c\}
 \subset 2^\omega$. The (additional) inductive assumptions,
  for $\beta+1,\gamma <\alpha < \mathfrak c$
are:
\begin{enumerate}
\item for each $t\in 2^{<\omega}$, $Q_t$ is $\tau_{\beta}$-dense,
\item   $\tau_\gamma$ and $\mathcal I_\gamma$ have cardinality
less than $\mathfrak c$,
\item each $I\in \mathcal I_\gamma$ is\label{sequences} a $\tau_\gamma$-converging sequence,
\item for each $I\in \mathcal I_\gamma$, $\varphi(I)\subset 2^{<\omega}$ is
either a chain or an anti-chain, 
\item if $r_\beta$ is in\label{frechet} the $\tau_{\beta+1}$-closure
of $A_\beta$, then there is an $I\in \mathcal I_{\beta+1}$ that converges
to $r_\beta$ and is contained in $A_\beta$,
\item the set $\{ \rho_\beta\restriction n : n\in\omega\}$ is almost
disjoint from $\varphi(I)$ for every $I\in \mathcal I_\beta$,
\item the set $B_\beta = \bigcup\{ Q_t \cap \{ q_\ell : \ell < f_\beta(|t|) \}: 
  t\in \{ \rho_\alpha\restriction n : n\in \omega\}\}$
is closed and discrete with respect to 
$\tau_{\beta+1}$.
\end{enumerate}

We start the recursion with $\tau_0$ as above and letting
$\mathcal I_0$ be a countable family of $\tau_0$-converging sequences
satisfying that, for every $I\in \mathcal I_0$, there is a $t\in 2^{<\omega}$
such that $I\subset Q_t$, and for every $q\in \mathbb Q$ and $t\in 2^{<\omega}$,
 there is an $I\in \mathcal I$ that converges to $q$ and is contained
 in $Q_t$. Let us note that so long as the members of
 $\mathcal I_0$ remain converging, this ensures that $Q_t$ is dense
 for every $t\in 2^{<\omega}$.

 Now let $\alpha < \mathfrak b$ and assume that,
 for every $\beta+1, \gamma < \alpha$,
 $\tau_{\gamma}, \rho_\beta$
 and $\mathcal I_\gamma$ have been chosen so 
 that the inductive hypotheses are satisfied.

In case $\alpha$ is a limit ordinal, then we set
 $\tau_\alpha = \bigcup \{ \tau_\gamma : \gamma < \alpha\}$
 and $\mathcal I_\alpha = \bigcup\{ \mathcal I_\gamma :\gamma < \alpha\}$.
 The easy verification of the inductive conditions is omitted. 
 \medskip

 Now suppose that $\alpha = \beta+1$. We first consider 
 the pair $(r_\beta, A_\beta)$. If $r_\beta$ is an element
 of $A_\beta^{(1)}$ with respect to the topology $\tau_\beta$,
  then choose any $J_\beta\subset A_\beta$ that converges
  to $r_\beta$. By passing to an infinite subset $I_\beta\subset
   J_\beta$, we can ensure that $\varphi(I_\beta)$ is either
   a chain (including singleton chains) or an anti-chain in $2^{<\omega}$. 
In this case we set $\mathcal I_{\alpha}=\{I_\beta\}\cup\mathcal I_\beta$,
and set, for later reference, $A$ to be the empty set. 
In the case that there is no such sequence from $A_\beta$ converging
to $r_\beta$, set $\mathcal I_\alpha=\mathcal I_\beta$ and 
set $A$ equal to $A_\beta$. 

Choose any function $\tilde f\in \omega^\omega$ such that
 $f_\alpha \leq \tilde f$  and such that $I\subset I(q,\tilde f)$
 for every $I\in \mathcal I_\alpha$ that converges to $q\in \mathbb Q$.

Now choose any $\rho_\beta\in 2^{<\omega}$ satisfying that
   $\varphi(I)\cap \{ \rho_\beta\restriction n : n\in\omega\}$
   is finite for all $I\in \mathcal I_\alpha$, and
   set $B_\beta  
 = \bigcup\{ Q_t \cap \{ q_\ell : \ell < f_\beta(|t|) \}: 
  t\in \{ \rho_\alpha\restriction n : n\in \omega\}\}$.

  Define $\tau_\alpha $ to be $\tau_\beta\cup \{ U^\alpha_n : n\in\omega\}$
   where $\{ U^\alpha_n : n\in\omega\}$ is the sequence
    $\{ U_n : n\in\omega\}$ obtained by applying
    Lemma \ref{mainLemma} with assignments
     $\tilde \tau =\tau_\beta$, $\mathcal I = \mathcal I_\alpha$,
     $\tilde f = f_\beta$,
     $B= B_\beta$,  $r=r_\beta$
     and $A$ as described above.

    Lemma \ref{mainLemma} ensures every element
    of $\mathcal I_\alpha$ is a $\tau_\alpha$-converging
    sequence
    and that $B_\beta$ is $\tau_\alpha$-closed and discrete. 
    Similarly, if $r_\beta$ is not in $A_\beta^{(1)}$ as 
    determined by $\tau_\beta$, then $A=A_\beta$ and
     $U^\alpha_0$ is a neighborhood of $r_\beta$ that
     is disjoint from $A_\beta$. On the other hand,
     if $r_\beta$ is in $A_\beta^{(1)}$, then
      we added $I_\beta$ to $\mathcal I_\alpha$.
This completes the induction construction of 
$\tau_{\mathfrak c} = \bigcup\{ \tau_\alpha : \alpha < \mathfrak c\}$.

  It follows by Lemma \ref{notH} that $\tau_{\mathfrak c}$ is not
  H-separable. It follows from inductive condition \ref{sequences}
  and the fact that $\{\mathcal I_\alpha : \alpha\in\mathfrak c\}$
  and $\{\tau_\alpha :\alpha \in\mathfrak c\}$ are chains, 
  that every member of $\bigcup\{\mathcal I_\alpha : \alpha < \mathfrak c\}$
  is a $\tau_{\mathfrak c}$-converging sequence.  
  It then also follows from inductive condition \ref{frechet},
   and the fact that $\{ (r_\beta,A_\beta) : \beta\in \mathfrak c\}$
   enumerates $\mathbb Q\times [\mathbb Q]^{\aleph_0}$,
    that $(\mathbb Q, \tau_{\mathfrak c})$ is Fr\'echet-Urysohn.
\end{proof}

To prove this next result we perform the same
length $\mathfrak b$ recursive construction but with no attention
paid to the Fr\'echet-Urysohn property.

\begin{theorem}
There is a countable space that is mH-separable\label{justwH}
but not H-separable.
\end{theorem}

\begin{proof}
In this proof we recursively  construct
our  increasing family $\{ \tau_\alpha : \alpha < \mathfrak b\}$
of clopen bases for topologies on $\mathbb Q$ together with 
the special  family
 $\{ \rho_\alpha : \alpha < \mathfrak b\}
 \subset 2^\omega$ 
 by also recursively constructing an increasing chain,
  $\{ M_\alpha : \alpha < \mathfrak b\}$,
 of elementary
 submodels of $H(\mathfrak c^+)$ as per Proposition \ref{getModels}.
 For each $\alpha <\mathfrak b$,
  let $\delta_\alpha$ equal the ordinal
    $M_\alpha\cap \mathfrak b$ and choose a function $f_{\delta_\alpha}\leq 
    \hat f_{\delta_\alpha}\in
    \omega^\omega$ 
      that mod finite dominates $M_\alpha\cap \omega^\omega$.   
     The elementarity will be
    critical for ensuring that each $I(q, \hat f_{\delta_\alpha})$ 
    captures enough information,
    combined with inductive hypotheses motivated by
    Proposition \ref{ringRule},
    to ensure that the final
    topology is mH-separable. We will also insist that $\hat f_{\delta_\alpha}\in M_{\alpha+1}$.

For $0<\beta  < \alpha$ we assume we are preserving
these inductive properties:
   
\begin{enumerate}
\item for each $t\in 2^{<\omega}$, $Q_t$ is $\tau_{\beta}$-dense,
\item $M_\beta$ is an elementary submodel  of $H(\mathfrak c^+)$
satisfying that $|M_\beta|<\mathfrak b$ and $M_\beta\cap \mathfrak b = 
 \delta_\beta\in \mathfrak b$,
 \item each of $ \{Q_t:t\in 2^{<\omega}\}$ and $\{f_\xi: \xi\in\mathfrak b\}$
  are elements of $M_0$,
\item  $\{ M_\xi : \xi< \beta\}$  is an  element  of $M_\beta$,
\item $f_{\delta_\xi} \leq \hat f_{\delta_\xi}
\in M_{\xi+1}\cap\omega^\omega$  mod finite dominates $M_\xi\cap \omega^\omega$,
\item   $\tau_\beta$ is a subset and an element of $M_{\beta}$,
\item $\rho_\beta\in 2^\omega\cap M_\beta \setminus \bigcup\{M_\xi : \xi<\beta\}$,
 \item the\label{H2} set $B_\beta = \bigcup\{ Q_t \cap \{ q_\ell : \ell < f_\beta(|t|) \}: 
  t\in \{ \rho_\beta\restriction n : n\in \omega\}\}$
is closed and discrete with respect to 
$\tau_{\beta}$,
 \item for\label{wH1} each $q\in W\in \tau_\beta$ and $\gamma<\beta$, there is a $
 q\in U\in \tau_\gamma$
 such that $ U\cap I(q,\hat f_{\delta_\gamma}) \setminus W$ is 
 covered by a finite union from $\{ B_\xi : \gamma < \xi \leq \beta\}$.
\end{enumerate}

We start the recursion with $\tau_0$ as above and letting
 $\{ Q_t : t\in 2^{<\omega}\}$ be an element of an
 elementary submodel $M_0$ of $H(\mathfrak c^+)$
 such that $M_0$ has cardinality less than $\mathfrak b$
 and $M_0\cap \mathfrak b = \delta_0 \in \mathfrak b$.
 Let us note that it is immediate that for each
  $q\in \mathbb Q$ and each $t\in 2^{<\omega}$, 
  there is a sequence $I\subset I(q,\hat f_{\delta_0})\cap Q_t$
  that converges to $q$. This fact, combined with 
  inductive hypotheses \ref{wH1} and \ref{H2},
   ensures that each $Q_t$ is $\tau_\beta$ dense for all $\beta < \mathfrak b$.

Here is the inductive step at stage $0<\alpha$. Choose
any elementary submodel $M_\alpha$ of $H(\mathfrak c^+)$
such that  the sets $ \{M_\beta :\beta < \alpha\}$, $\{\hat f_{\delta_\beta} : \beta < \alpha\}$,
and $\{\tau_\beta : \beta <\alpha\}$ are elements and such that $|M_\alpha | < \mathfrak b$
and $\alpha\subset M_\alpha\cap \mathfrak b =\delta_\alpha$. 
By elementarity $\bigcup\{ M_\beta : \beta < \alpha\}$ is
an element of $M_\alpha$ and additionally
the set $2^\omega\cap M_\alpha \setminus
\bigcup\{ M_\beta : \beta < \alpha\}$ is not empty.
Choose any $\rho_\alpha \in  2^\omega\cap M_\alpha \setminus
\bigcup\{ M_\beta : \beta < \alpha\}$.  
Since $\alpha\in M_\alpha$, it follows that $f_\alpha $,
and therefore  $B_\alpha =
\bigcup\{ Q_t \cap \{ q_\ell : \ell < f_\alpha(|t|) \}: 
  t\in \{ \rho_\alpha\restriction n : n\in \omega\}\}$,
  are elements of $M_\alpha$. 
Also choose any $\tilde \delta\in M_\alpha\cap \mathfrak b$
   so that $\delta_\beta < \tilde \delta$ for all 
    $\beta  < \alpha$.

  Define
  $\tau_\alpha$ to be $\bigcup\{\tau_\beta :\beta < \alpha \}\cup
  \{U^\alpha_n : n\in \omega\}$ where $\{U^\alpha_n : n\in\omega\}$
  is obtained by applying Lemma \ref{mainLemma}, working within
   $M_\alpha$, to the
  values $\tilde\tau = \bigcup\{\tau_\beta:\beta < \alpha\}$,
  $\mathcal I = \emptyset$,
   $A=\emptyset$, $r=0$, $\tilde f = f_{\tilde\delta}  $,
    $B = B_\alpha$.  All we needed is that
     $\{ U^\alpha_n : n\in \omega\}\in M_\alpha$ witnesses that
      $B_\alpha$ is closed and discrete and that,
      for each $q\in U^\alpha_n $ and $\beta < \alpha$,
       $ I(q,f_{\delta_\beta})\setminus U^\alpha_n$ is covered by $B_\alpha$. 
This completes the construction, and we set $\tau_{\mathfrak b}$
equal to $\bigcup\{ \tau_\alpha : \alpha < \mathfrak b\}$.

By Lemma \ref{notH},  it suffices to prove that $\tau_{\mathfrak b}$
is mH-separable.  Suppose that $\{ D_n : n\in\omega\}$ is a 
descending sequence of $\tau_{\mathfrak b}$-dense subsets of
 $\mathbb Q$. Clearly, for each $n\in\omega$ and $\beta < \mathfrak b$,
  $D_n$ is also $\tau_{\beta}$-dense. Fix any $q\in \mathbb Q$,
  we prove there is a sequence of finite sets $H^q_n\subset D_n$
  so that every $\tau_{\mathfrak b}$-neighborhood of $q$
  hits all but finitely many members of $\{ H^q_n : n\in\omega\}$. 

  Fix a continuous increasing chain, $\{ M^q_\alpha : \alpha < \mathfrak b\}$,
   of elementary submodels of $H(\mathfrak c^+)$ satisfying that
   $\{D_n :n\in\omega\}$ and 
     $\{ M_\beta, \tilde f_{\delta_\beta}
     : \beta < \mathfrak b\}$ are elements of $ M^q_0$ and for all
      $\alpha<\mathfrak b$, $M^q_\alpha$ has cardinality less
      than $\mathfrak b$ and $\mu_\alpha = M^q_\alpha\cap \mathfrak b$
      is an element of $\mathfrak b$. 
 An increasing chain being continuous simply means that for
limit $\gamma<\mathfrak b$,  
$M^q_\gamma$ is equal to $\bigcup\{ M^q_\alpha :\alpha < \gamma\}$.
This implies that the set $\{ \mu_\alpha : \alpha <\mathfrak b\}$
is a closed and unbounded subset of (the ordinal) $\mathfrak b$.

By standard results, there is a closed and unbounded
subset $C$ of $\mathfrak b$ satisfying that
 $\mu_\zeta  = \zeta$ for all $\zeta\in C$.
Consider any $\zeta\in C$.
    Since $\{M_\beta : \beta < \mathfrak b\}$
is an element of $M^q_\zeta$, it follows that $\{M_\beta : \beta < \zeta\}$
is a subset of $M^q_\zeta$. Moreover, we verify that $M^q_\zeta\cap
 \bigcup \{ M_\beta : \beta < \mathfrak b\} = 
 \bigcup \{ M_\beta : \beta < \zeta\}$.  First, for each 
  $\beta < \zeta$, $M_\beta\in M^q_\zeta$ and, since 
  $M^q_\zeta \models |M_\beta| \in \mathfrak b$, 
  it follows that $M^q_\zeta\models |M_\beta|\subset M^q_\zeta$.
  These two facts combined implies that $M_\beta\subset M^q_\zeta$. 
  On the other hand if $z\in M^q_\zeta\cap \bigcup\{ M_\beta : \beta < 
  \mathfrak b\}$, then $M^q_\zeta \models (\exists \beta <\mathfrak b)~~
  z\in M_\beta$. This implies that $M^q_\zeta\cap 
  \bigcup\{ M_\beta : \beta < \mathfrak b\}
$ is contained in $\bigcup\{ M_\beta : \beta < \zeta\}$.

Now define $D = \bigcup \{ D_n\cap U(q,n)\setminus U(q,n+1) : n\in\omega\}$
and note that $D\in M^q_0$. Since $\{ U(q,n)\setminus U(q,n+1) : n\in \omega\}$
is a partition of $\mathbb Q\setminus \{q\}$ into $\tau_{\mathfrak b}$-clopen
sets and, for each $n$,  $D_n\cap U(q,n)\setminus U(q,n+1)$ is dense
in $U(q,n)\setminus U(q,n+1)$, 
  $D$ is a $\tau_{\mathfrak b}$-dense set.  
 Choose any strictly increasing sequence $\{ \zeta_k :
 k\in \omega\}$ from $C$ and let  $\zeta\in C$ be the supremum
  of $\{\zeta_k : k\in\omega\}$.

Suppose first there is a $\rho\in 2^\omega$ such
that $q$ is in the $\tau_{\mathfrak b}$-closure of $D\cap \bigcup\{Q_t : t\subset \rho\}$.
By elementarity, then we may choose  a $k\in\omega$
and a $\rho\in M^q_{\zeta_k}$  such that $q$ is in the $\tau_{\zeta_k}$-closure
of $D\cap \bigcup\{Q_t : t\subset \rho\}$. 
Since $\tau_{\zeta_k}$ has cardinality less than $\mathfrak b$,  the 
standard proof that shows such spaces are $mH$-separable, shows
that   there is a sequence $\{ F^q_n \subset U(q,n)
: n\in\omega\}\in M^q_{\zeta_{k+1}}$
 of finite subsets of $D\cap \bigcup\{Q_t : t\subset \rho\}$ 
 satisfying that, for each $q\in U\in \tau_{\zeta_k}$,
  $U\cap H^q_n$ is not empty for all but finitely many $n\in\omega$.
  Similarly, there is a  function $g_k\in \omega^\omega\cap M^q_{\zeta_{k+1}}$,
  such that, for each $n\in\omega$, $F^q_n\subset \{ q_\ell : \ell < g_k(n)\}$. 
  Since there are $\xi \in M^q_{\zeta_{k+1}}$ such that 
   $\{ n \in\omega : g_k(n) < f_{\xi }(n)\}$ is infinite, 
   the set $L = \{ n : g_k(n) < \hat f_{\zeta_{k+1}}(n)\}$ is infinite.
   It follows that, for all $n\in L$, $F^q_n \subset I(q,\hat f_{\zeta_{k+1}}(n)\}$.
   Now simply choose a strictly increasing function $g\in\omega^\omega$
   satisfying that, for all $n\in \omega$, there is an $n\leq\bar n\in L$ such
   that $\hat f_{\zeta_{k+1}}(\bar n) \leq g(n)$.  Now set, for every
   $n\in\omega$, $H^q_n = D_n  \cap I(q,\hat f_{\zeta_{k+1}})\cap \{ q_\ell
     : \ell <g(n)\}$. 
     It is apparent that, for all $U\in \tau_{\zeta_k}$, $U\cap H^q_n$ is 
     non-empty for all but finitely many $n\in\omega$. On the other 
     hand,   
since $\rho\neq \rho_\beta$
  for all $\zeta_k \leq \beta < \mathfrak b$, it follows from item \ref{wH1} 
  that $W\cap H^q_n$ is not empty for all but finitely many $n\in\omega$
  and $q\in W\in \tau_{\mathfrak b}$.
\medskip

  Therefore for the remainder of the proof we assume that $q$ is not in
  the $\tau_{\mathfrak b}$-closure of $D\cap 
  \bigcup\{Q_t : t\subset \rho\}$ for each $\rho\in 2^\omega$.  
  Note that it then follows that for each $q\in U\in \tau_{\mathfrak b}$, 
   the set $\{ t : U\cap D \cap Q_t\neq\emptyset \} $ has arbitrarily large
   finite anti-chains.
\medskip

 Fix any $k\in\omega$   and, working in $M^q_{\zeta_k}$, we note
 that   for each
   $q\in U\in \tau_{\zeta_k}$, 
    there is a strictly increasing function $g_U\in \omega^\omega
   \cap M^q_{\zeta_k}$ satisfying that for each $n$
   there are $n\leq m\leq  \ell_1<\ell_2<\cdots<\ell_n < g_U(n)$ such that,
   for each $1\leq i\leq n$
\begin{enumerate}
\item  $q_{\ell_i} \in U\cap  D \cap U(q,n)\setminus U(q,g_U(n))$, and
\item   $\{ t\in 2^{<\omega} : \{q_{\ell_i} : 1\leq i\leq n\}\cap Q_t\neq\emptyset\}$
is an anti-chain of cardinality $n$. 
\end{enumerate}  
For each $k,n\in\omega$ and $U\in \tau_{\zeta_k}$, let $H(U,k,n)$ 
denote the set $\{ q_{\ell_1},\ldots, q_{\ell_n}\}$ chosen as in
the definition of $g_U$. Observe that $H(U,k,n)\subset I(q, g_U)$.

Choose any $g_k\in \omega^\omega\cap M^q_{\zeta_{k+1}}$ satisfying
that $g_U <^* g_k $ for all $U\in \tau_{\zeta_k}$. Let $L_k  =\{n: 
  g_k(n) < \hat f_{\zeta}(n)\}$ and note as above that $L_k$ is infinite. 
  Consider any $n\in\omega$ such that $g_U(n) \leq g_k(n) \leq \hat f_{\zeta}(n)$
  and again note that $H(U,k,n)\subset D_n\cap I(q,\hat f_\zeta)$.   
  
  Finally choose any $g\in \omega^\omega$ so large that, for each $k\in\omega$,
  for all but finitely many $n\in\omega$, there is a $n\leq \bar n \in L_k$
  such that $\hat f_\zeta (\bar n) \leq g(n)$. 
  Set, for all $n\in\omega$,  $H^q_n = I(q, \hat f_\zeta)\cap D_n \cap \{q_\ell : \ell < g(n)\}$.
  It then follows that, 
  for all $k\in\omega$ and $U\in \tau_{\zeta_k}$ (i.e. for all $U\in \tau_{\zeta}$),
  for all but finitely many $n\in\omega$,  $H(U,k,n)\subset H^q_n$.

  Finally we prove that $\{ H^q_n : n\in\omega\}$ witnesses that the $mH$-property
  with respect to the sequence $\{D_n : n\in\omega\}$ holds at the point $q$.

  Consider an $q\in W\in\tau_{\mathfrak b}$.
By induction hypothesis \ref{wH1},
 we may choose some $U \in \bigcup_{\beta<\zeta}\tau_\beta$
 so that there is a
 finite set $ S\subset  \mathfrak b $ 
 such that $U  \cap I(q,\hat f_{\delta_\zeta})\setminus W$
 is covered by $\bigcup\{ B_\xi : \xi \in S\}$.
 By possibly having to shrink $U$ we may assume that
  $\zeta \leq \min(S)$. 
  Choose $k\in\omega$ so that $U\in \tau_{\zeta_k}$ 
  and   any $n\in\omega$ such that $|S|< n$
  and $H(U,k,n)\subset U\cap  H^q_n$. Let $T(U,k,n)$ denote
  the anti-chain of size $n$ satisfying that
     $H(U,k,n)\cap Q_t\neq\emptyset$ for each $t\in T(U,k,n)$. 
 Clearly there is a $t\in T(U,k,n)$ that is not in
 the set $\{ \rho_\xi\restriction j : j\in \omega,\ \xi\in S\}$
 because this latter set is a union of $|S|$ many chains. 
 Therefore the non-empty set $U\cap Q_t\cap H^q_n\subset I(q,\hat f_{\zeta})$
 is a subset
 of $W$. This proves that $W\cap H^q_n$ is non-empty for all but
 finitely many $n\in\omega$.
 \end{proof}

This next construction  is   the same as  in 
the proof of Theorem \ref{justwH} but it uses
the fact that countable spaces of character
less $\mathfrak p$ are Fr\'echet-Urysohn and
the same properties related to the $I(q,f)$'s to
preserve large enough families of converging sequences. 

\begin{theorem}[$\mathfrak p = \mathfrak b$]
There is a countable Fr\'echet-Urysohn space that is
 not H-separable.
\end{theorem}

\begin{proof} For the reader's convenience we repeat
the main details but also introduce some minor changes at limit steps.
We recursively  construct
our  increasing family $\{ \tau_\alpha : \alpha < \mathfrak b\}$
of clopen bases for topologies on $\mathbb Q$ together with 
the special  family
 $\{ \rho_{\alpha+1} : \alpha < \mathfrak b\}
 \subset 2^\omega$ 
 by also constructing an increasing chain,
  $\{ M_\alpha : \alpha < \mathfrak b\}$,
 of elementary
 submodels of $H(\mathfrak c^+)$ as per Proposition \ref{getModels}.
 For this construction, we choose to make $\{ M_\alpha :\alpha <\mathfrak b\}$
  a continuous increasing chain,  and for limit $\alpha$,
    $\tau_\alpha $ will equal $\bigcup\{ \tau_\beta : \beta < \alpha\}$. 
Similarly for limit $\alpha$ our choice for $B_\alpha$ will be the empty set.

 For each $\alpha <\mathfrak b$,  
  let $\delta_\alpha$ be the ordinal
    $M_\alpha\cap \mathfrak b$  and again choose $\hat f_{\delta_\alpha}
    \geq f_{\delta_\alpha}$
    in $M_{\alpha+1}\cap \omega^\omega$ so as to 
    mod finite dominate $M_\alpha\cap \omega^\omega$.
  The elementarity will be
    critical for ensuring that each $I(q,\hat f_{\delta_\alpha})$ 
    captures enough information,
    combined with inductive hypotheses motivated by
    Proposition \ref{ringRule},
    to ensure that the final
    topology is Fr\'echet-Urysohn.

For $0<\gamma  < \alpha$
 and $\beta+1 < \alpha$ we assume we are preserving
these inductive properties:
   
\begin{enumerate}
\item for each $t\in 2^{<\omega}$, $Q_t$ is $\tau_{\gamma}$-dense,
\item $M_\gamma$ is an elementary submodel  of $H(\mathfrak c^+)$
satisfying that $|M_\gamma|<\mathfrak b$ and $M_\gamma\cap \mathfrak b = 
 \delta_\gamma\in \mathfrak b$,
 \item each of $ \{Q_t:t\in 2^{<\omega}\}$ and $\{f_\xi: \xi\in\mathfrak b\}$
  are elements of $M_0$,
\item  $\{ M_\xi : \xi\leq \beta\}$ is an  element  of $M_{\beta+1}$,
\item $\hat f_{\delta_\beta}$ is an element of $M_{\beta+1}\cap \omega^\omega$
that is strictly increasing and mod finite dominates $M_\beta\cap \omega^\omega$,
\item   $\tau_\beta$ is a subset of $M_\beta$  and $\tau_{\beta+1}\setminus \tau_\beta\in M_{\beta+1}$ is countable,
\item $\rho_{\beta+1}\in 2^\omega\cap M_{\beta+1} \setminus M_\beta $,
 \item the\label{H3} set $B_{\beta+1} = \bigcup\{ Q_t \cap \{ q_\ell : \ell < f_{\beta}(|t|) \}: 
  t\in \{ \rho_{\beta+1}\restriction n : n\in \omega\}\}$
is closed and discrete with respect to 
$\tau_{\beta+1}$,
 \item for\label{mH1} each $q\in W\in \tau_\gamma$ and $\mu<\gamma$, there is a
 $q\in U\in \tau_\mu$
 such that $ U\cap I(q,\hat f_{\delta_\mu}) \setminus W$ is 
 covered by a finite union from $\{ B_{\xi+1} : \mu \leq \xi < \gamma\}$.
\end{enumerate}

We choose $M_0$ exactly as in the proof of
 Theorem \ref{justwH}. For a limit ordinal $\alpha < \mathfrak b$,
   we set $ M_\alpha = \bigcup\{ M_\beta : \beta < \alpha\}$ and
      $\tau_\alpha = \bigcup\{ \tau_\beta : \beta < \alpha\}$ and there
      is nothing more to do.

      Now assume that $\alpha = \beta+1$.  
We choose   any elementary submodel $M_\alpha$ of $H(\mathfrak c^+)$
such that $ \{M_\beta , \tau_\beta\}\in M_\alpha$
 and such that $|M_\alpha | < \mathfrak b$
and $M_\alpha\cap \mathfrak b =\delta_\alpha>\alpha$. 
We also choose $\hat f_{\delta_\beta}\geq f_{\delta_\beta}$
to be any element
of $M_\alpha\cap \omega^\omega$ that is strictly increasing
and that mod finite dominates
 $M_\beta\cap \omega^\omega$.
Choose any $\rho_\alpha \in  2^\omega\cap M_\alpha \setminus M_\beta$ 
and set  $B_\alpha =
\bigcup\{ Q_t \cap \{ q_\ell : \ell < f_\alpha(|t|) \}: 
  t\in \{ \rho_\alpha\restriction n : n\in \omega\}\}$.
  The set $B_\alpha$ is an element of $M_\alpha$.  
  Define
  $\tau_\alpha$ to be $\tau_\beta \cup
  \{U^\alpha_n : n\in \omega\}$ where $\{U^\alpha_n : n\in\omega\}$
  is obtained by applying Lemma \ref{mainLemma}, working within
   $M_\alpha$, to the
  values $\tilde\tau = \tau_\beta$,
  $\mathcal I = \emptyset$,
   $A=\emptyset$, $r=0$, $\tilde f = \hat f_{\delta_\beta}+f_\alpha  $, and
    $B = B_\alpha$.

This completes the construction, and we set $\tau_{\mathfrak b}$
equal to $\bigcup\{ \tau_\alpha : \alpha < \mathfrak b\}$.  
 It follows by Lemma \ref{notH} that $\tau_{\mathfrak b}$ is not
  H-separable. Now
we prove that $(\mathbb Q, \tau_{\mathfrak b})$ is 
 Fr\'echet-Urysohn.

 Consider any $A\subset \mathbb Q$
 and $\tau_{\mathfrak b}$-limit point $q$ of $A$.   
 We again note that $q$ is a $\tau_\alpha$-limit point 
 of $A$ for every $\alpha < \mathfrak b$. 
Fix a continuous increasing chain, $\{ M^q_\alpha : \alpha < \mathfrak b\}$,
   of elementary submodels of $H(\mathfrak c^+)$ satisfying that
   $A$ and 
     $\{ M_\beta: \beta < \mathfrak b\}$ are elements of $ M^q_0$ and for all
      $\alpha<\mathfrak b$, $M^q_\alpha$ has cardinality less
      than $\mathfrak b$ and $\mu_\alpha = M^q_\alpha\cap \mathfrak b$
      is an element of $\mathfrak b$.  Let $C$ be a closed 
      and unbounded set of $\zeta\in \mathfrak b$ that satisfy
       $M^q_\zeta \cap \mathfrak b = \zeta$. Choose a strictly
       increasing sequence $\{ \zeta_k : k\in\omega\}\subset C$,
        and let $\mu$ be the supremum. Note also that, for $k\in\omega$,
         $M^q_{\zeta_k}\subset M^q_\mu$ and that 
        $\delta_\mu = \mu$.

For each $k\in\omega$, using that $\tau_{\zeta_k}$ is a
 Fr\'echet-Urysohn topology, choose a sequence $J_k\subset A$
 that is an element of $M^q_{\zeta_{k+1}}$ and that 
 $\tau_{\zeta_k}$-converges to $q$.    
Let  $L_k = \{ n\in  \omega : J_k\cap U(q,n)\setminus U(q,n+1)\neq\emptyset\}$
 and define the  function 
 $g_k\in\omega^\omega\cap M^q_{\zeta_{k+1}}$ 
 according to the rule that for each
  $n\in L_k$, 
  there is an $\ell< g_k(n)$ such
that $q_\ell \in I_k\cap U(q,n)\setminus U(q,n+1)$
and for each  $m\in\omega\setminus L_k$,
  $g_k(m) = g_k(\min(L_k\setminus m))$.  
By elementarity, there is a $\xi_k\in \mathfrak b
\cap M^q_{\zeta_{k+1}}$ such that $f_{\xi_k}(n)> g_k(n)$
for infinitely many $n$. It follows that $\hat f_{\delta_{\zeta_{k+1}}}(n)
> g_k(n)$ for infinitely many $n$. Since $\hat f_{\delta_{\zeta_{k+1}}}$ is
a strictly increasing function, it follows that for each $n$
such that $\hat f_{\delta_{\zeta_{k+1}}}(n)
> g_k(n)$, we also have that $g_k(\min(L_k\setminus n)) < 
\hat f_{\delta_{\zeta_{k+1}}}(\min(L_k\setminus n))$. 
This implies that 
$I(q,\hat f_{\delta_{\zeta_{k+1}}})\cap J_k$ is infinite.
Since $\delta_{\zeta_{k+1}}<\delta_\mu=\mu$,
 it similarly follows that, for each $k\in\omega$,
    $I(q,\hat f_{\mu}) \cap J_k$ is infinite.  

    \begin{claim}  $q$ is in the $\tau_\mu$-closure of
       $A\cap I(q,\hat f_{\mu})$. 
    \end{claim}

  \bgroup

\def\proofname{Proof of Claim.}

\begin{proof} Let $U$ be any element of $\tau_\mu$. 
  Choose $k\in\omega$ so that $U\in \tau_{\zeta_k}$
   and observe that  since $U$ contains a cofinite subset of $J_k$,
  $   U\cap (J_k\cap I(q,\hat f_{\mu}))\subset U\cap A\cap I(q,\hat f_\mu)$ is infinite.
\end{proof}

Since $|\tau_\mu|<\mathfrak p$ there are sequences contained
 in $A\cap I(q,\hat f_\mu)$ that $\tau_\mu$-converge to $q$. Now fix any $\beta\in \mu$
 and, trivially, observe that ``there exists some $\mu>\beta$ satisfying
 that there exists some sequence contained in $A\cap I(q,f_{\delta_\mu})$
 that  $\tau_{\mu}$-converges to $q$''. 
 Now, working in $M^q_\mu$, apply elementarity,
  and choose, for each $k\in\omega$, some $\zeta_k < \mu_k < \mu$ such
  that there is sequence $I_k\subset A\cap I(q,\hat f_{\delta_{\mu_k}})$ 
  that $\tau_{\mu_k}$-converges to $q$. By passing to subsequences,
   we can assume that, for each $k$, 
     $\varphi(I_k)$ is one of: a singleton $t_k\in 2^{<\omega}$,
      an infinite anti-chain $T_k\subset 2^{<\omega}$, or
      an infinite subset of $\{ \psi_k\restriction n : n\in\omega\}$
      for some $\psi_k\in 2^\omega\cap M^q_\mu$. 

 \begin{claim} If, for some $k\in\omega$, $\varphi(I_k)$ is either a singleton
  or an anti-chain, then $I_k$ converges to $q$ with respect to $\tau_{\mathfrak b}$. 
 \end{claim}

 \begin{proof}  
 Consider any $W\in \tau_{\mathfrak b}$.  By inductive condition \ref{mH1},
  there is   $U\in \tau_{\mu_k}$ and a finite set $\{\xi_i : i<\ell\}\subset
     \mathfrak b \setminus \mu_k$ satisfying that
       $U\cap I(q,\hat f_{\mu_k}) \setminus W$ is contained in
         $\bigcup\{ B_{\xi_i+1} : i< \ell\}$.  
Of course $\bigcup\{ \varphi(B_{\xi_i+1}) : i<\ell\}$ is a subset of
   $\bigcup \{ \rho_{\xi_i+1}\restriction n : i< \ell,\ \ n\in\omega\}$. 
The set $\bigcup\{ B_{\xi_i+1} : i< \ell\}$ is almost disjoint
from every $Q_t$ ($t\in 2^{<\omega}$) and for each
 $t\in 2^{<\omega}\setminus 
  \bigcup \{ \rho_{\xi_i+1}\restriction n : i< \ell,\ \ n\in\omega\}$,
  it 
      is disjoint from
      $Q_t$.   
      It therefore follows that $I_k\cap 
        \bigcup\{ B_{\xi_i+1} : i< \ell\} $ is finite. 
        Since $U$ contains $I_k$ mod finite, it follows
        that $W$ does as well.
 \end{proof}

 The remaining case is that, for each $k$, $\varphi(I_k)$
 is an infinite chain contained in $\{ \psi_k\restriction n : n\in \omega\}$
 with $\psi_k\in M^q_\mu$.  
 
 \begin{claim} If for some $k\in\omega$, $q$  is in the $\tau_{\mu}$-closure
 of $I_k$, then $I_k$ has a subsequence that $\tau_{\mathfrak b}$-converges to $q$.
 \end{claim}

 \begin{proof}  If $\psi_k\in \{ \rho_{\xi+1} : \xi <\mathfrak b\}$, then,
  by elementarity, 
 choose $\xi_k\in M^q_\mu$ so that $\psi_k = \rho_{\xi_k+1}$.  Choose 
  $ k +1 < \bar k\in \omega$ so that $\xi_k < \mu_{\bar k}$. 
  By assumption $I_k \subset I(q, \hat f_{\delta_{\mu_{\bar k}}})$ and
     $q$ is a $\tau_{\mu_{\bar k}}$-limit of $I_k$. Choose any infinite
      $I'_k\subset I_k$ that converges to $q$ with respect to 
      $\tau_{\mu_{\bar k}}$. Note that $I'_k$ is almost disjoint
      from $B_{\gamma+1}$ for all $\mu_{\bar k}\leq \gamma$. 
      Therefore, by induction hypothesis \ref{mH1}, $I_k'$ 
      converges to $q$ with respect to $\tau_{\mathfrak b}$. 
 \end{proof}

  Therefore, we now assume that $q$ is not in the $\tau_\mu$-closure
  of $I_k$ for each $k\in \omega$, and, as a consequence we obtain
  this next claim.

  \begin{claim} The
  set $\{ \psi_k : k\in\omega\}$ is infinite.
  \end{claim}

  \begin{proof} Assume that there is a single $\psi\in 2^\omega$
  such that $\psi_k = \psi$ for infinitely many $k\in \omega$. 
  As in the previous proof, choose $\bar k$ so that, if 
   $\psi \in \{\rho_\xi : \xi\in\mathfrak b\}$, then
      $\psi \in \{ \rho_\xi : \xi < \mu_{\bar k}\}$. 
   Choose any $k >\bar k$ so that $\psi = \psi_k$. Now, 
   we have that for all $\mu_{ k} < \beta < \mu$,
   $U^\beta_{m_q} $ mod finite contains
    $ I(q,\hat f_{\delta_\beta}) \setminus B_{\beta+1}$, 
    where $m_q$ is chosen so that $q\in U^\beta_{m_q}$. 
    Since $I_k$ is mod finite contained in 
       $I(q,\hat f_{\delta_\beta})\setminus B_{\beta+1}$
       (because $\rho_{\beta+1}\neq \psi$) 
       it would follow that $I_k$ converges to $q$ with 
        respect to $\tau_\mu$.
  \end{proof}

By the previous claim we can, by re-indexing an infinite subsequence
 of $\{ \psi_k : k\in \omega\}$, assume that they are pairwise
 distinct and converge to some $\psi\in 2^\omega$ and that
    $\{ n_k : k\in\omega\}$ strictly increases where
     $n_k = \min\{ n : \psi(n)\neq \psi_k(n)\}$. 
  Next, we pass
 to a cofinite subset of each $I_k$ and thereby assume that
 the sets  $\{ \varphi(I_k) : k\in \omega\}$ are pairwise disjoint
 by simply ensuring that $|t|>n_k$ for all $ t\in \varphi(I_k)$.

 \begin{claim} For every $U\in \tau_\mu$, $I_k\subset^* U$ for all
 but finitely many $k\in \omega$.
 \end{claim}

 \begin{proof}  For each $U\in \tau_{\mu_k}$, $I_\ell\subset ^* U$
 for all $k<\ell$ since $I_\ell$ converges to $q$ with respect
 to $\tau_{\mu_\ell}$ and $U\in \tau_{\mu_\ell}$. 
 \end{proof}

  \egroup

  Finally, since $|\tau_\mu|<\mathfrak b$, there is a selection 
   $S\subset \bigcup\{ I_k : k\in \omega\}$  ($|S\cap I_k| = 1$
   for all $k\in \omega$)  satisfying that 
    $S\subset^* U$ for all $U\in \tau_\mu$.
We have produced a sequence $S$ that converges to $q$ with respect
to $\tau_\mu$ such that $\varphi(S)$ is an infinite anti-chain.
Clearly $B_{\beta+1}\cap S$ is finite for all $\mu\leq\beta<\mathfrak b$,
 and, by induction hypothesis \ref{mH1}, this completes the proof.
\end{proof}

\end{document}